\documentclass{article}%
\usepackage{amsfonts}
\usepackage{graphicx}
\usepackage{amsmath}%
\usepackage{amssymb}
\newtheorem{theorem}{Theorem}

\newtheorem{lemma}[theorem]{Lemma}

\newtheorem{proposition}[theorem]{Proposition}

\newenvironment{proof}[1][Proof]{\textbf{#1.} }{\ \rule{0.5em}{0.5em}}
\begin{document}

\title{The Fr\"{o}licher-Nijenhuis Calculus in Synthetic Differential Geometry}
\author{Hirokazu Nishimura\\Institute of Mathematics, University of Tsukuba\\Tsukuba, Ibaraki, 305-8571\\Japan}
\maketitle

\begin{abstract}
Just as the Jacobi identity of vector fields is a natural consequence of the
general Jacobi identity of microcubes in synthetic differential geometry, it
is to be shown in this paper that the graded Jacobi identity of the
Fr\"{o}licher-Nijenhuis bracket is also a natural consequence of the general
Jacobi identity of microcubes.

\end{abstract}

\section{Introduction}

It has long been known that the totality of vector fields on a well-behaved
space forms a Lie algebra. Since vector fields and their corresponding
derivations can not be identified in synthetic differential geometry, it is by
no means direct to establish this fact synthetically. It was Nishimura
\cite{n1} that noted, behind the Jacobi identity of vector fields, what is to
be called the general Jacobi identity of microcubes.

The Fr\"{o}licher-Nijenhuis bracket, discussed in \cite{fn} and \cite{n0}, is
a natural extension of the Lie bracket of vector fields to
tangent-vector-valued differential forms. The principal objective in this
paper is to derive the graded Jacobi identity for the Fr\"{o}licher-Nijenhuis
bracket from the general Jacobi identity synthetically. The interior
derivation and the Lie derivation are discussed in passing.

\section{Preliminaries}

We assume that the reader is familiar with Lavendhomme's textbook \cite{l1} on
synthetic differential geometry up to Chapter 5. We denote by $D$ the subset
of $\mathbb{R}$ (the extended set of real numbers satisfying the Kock-Lawvere
axiom) consisting of elements $d$ of $\mathbb{R}$ with $d^{2}=0$. Given a
function $F:D\rightarrow\mathbf{E}$ of $D$ into a Euclidean space $\mathbf{E}%
$, we write $\mathbf{D}F$ for the entity of $\mathbf{E}$\ characterized by
\[
F(d)=F(0)+d\mathbf{D}F
\]
for any $d\in D$.

Given a microlinear space $M$, we denote $M^{D}$ by $\mathbf{T}M$. The notion
of strong difference $\overset{\cdot}{-}$ was introduced by Kock and
Lavendhomme \cite{kl} into synthetic differential geometry. The following
proposition belongs to the folklore.

\begin{proposition}
For any function $f:M\rightarrow N$ of microlinear spaces and any $\gamma
_{1},\gamma_{2}\in M^{D^{2}}$ with $\gamma_{1}\mid_{D(2)}=\gamma_{2}%
\mid_{D(2)}$, we have
\[
f\circ\gamma_{1}\overset{\cdot}{-}f\circ\gamma_{2}=f\circ\left(  \gamma
_{1}\overset{\cdot}{-}\gamma_{2}\right)
\]

\end{proposition}

The notion of strong difference $\overset{\cdot}{-}$ can be relativized. Since
$M^{D^{3}}=(M^{D})^{D^{2}}$, microcubes on $M$ can be viewed as microsquares
on $M^{D}$. According to which $D$ in the right-hand side of $D^{3}=D\times
D\times D$ appears as the superscript just over $M$, we get the three
relativized strong differences $\underset{i}{\overset{\cdot}{-}}$ $(i=1,2,3)$,
for which we have the following general Jacobi identity.

\begin{theorem}
Let $\gamma_{123},\gamma_{132},\gamma_{213},\gamma_{231},\gamma_{312}%
,\gamma_{321}\in M^{D^{3}}$. As long as the following three expressions are
well defined, they sum up only to vanish:
\begin{align*}
&  (\gamma_{123}\overset{\cdot}{\underset{1}{-}}\gamma_{132})\overset{\cdot
}{-}(\gamma_{231}\overset{\cdot}{\underset{1}{-}}\gamma_{321})\\
&  (\gamma_{231}\overset{\cdot}{\underset{2}{-}}\gamma_{213})\overset{\cdot
}{-}(\gamma_{312}\overset{\cdot}{\underset{2}{-}}\gamma_{132})\\
&  (\gamma_{312}\overset{\cdot}{\underset{3}{-}}\gamma_{321})\overset{\cdot
}{-}(\gamma_{123}\overset{\cdot}{\underset{3}{-}}\gamma_{213})
\end{align*}

\end{theorem}

The theorem was established by Nishimura \cite{n1} and has been reproved twice
by himself in \cite{n2} and \cite{n3}.

We use the notion of \textit{linear connection} in the sense of Definition 1
in \S \S 5.1 of Lavendhomme \cite{l1}. Given a linear connection $\nabla$ on a
microlinear space $M$ and a linear connection $\nabla^{\prime}$ on a
microlinear space $N$ with a function $f:M\rightarrow N$, we say that
$\nabla^{\prime}$ \textit{is }$f$\textit{-related to} $\nabla$ provided that
\[
f\circ\nabla(t_{1},t_{2})=\nabla^{\prime}(f\circ t_{1},f\circ t_{2})
\]
for any $t_{1},t_{2}\in\mathbf{T}M$ with $t_{1}(0)=t_{2}(0)$. We will often
write $\nabla\gamma$ for $\nabla(t_{1},t_{2})$, where $t_{1}=\gamma(\cdot,0)$
and $t_{2}=\gamma(0,\cdot)$.

We write $\mathbb{S}_{n}$ for the permutation group of the first $n$ natural
numbers, namely, $1,...,n$. Given $\gamma\in M^{D^{n}}$ and $\sigma
\in\mathbb{S}_{n}$, we define $\gamma^{\sigma}\in M^{D^{n}}$ to be
\[
\gamma^{\sigma}(d_{1},...,d_{n})=\gamma(d_{\sigma(1)},...,d_{\sigma(n)})
\]
for any $(d_{1},...,d_{n})\in D^{n}$. Given $\gamma\in M^{D^{n}}$, we write
$\mathfrak{o}_{n}(\gamma)$ for $\gamma(0,...,0)$.

\section{Tangent-Vector-Valued Differential Forms}

It is well known in synthetic differential geometry that vector fields can be
viewed in three different but essentially equivalent ways, namely, as sections
of the tangent bundle, as infinitesimal flows and as infinitesimal
transformations, for which the reader is referred to \S 3.2 of Lavendhomme
\cite{l1}. These three viewpoints can easily be extended to
tangent-vector-valued differential forms. Let $M$ be a microlinear space. The
first orthodox viewpoint is to regard tangent-vector-valued differential
$p$-forms on $M$ as mappings $K:M^{D^{p}}\rightarrow M^{D}$ with
$\mathfrak{o}_{p}(\gamma)=\mathfrak{o}_{1}(K(\gamma))$ for any $\gamma\in
M^{D^{p}}$ and satisfying the $p$-homogeneity and the alternating \ property
in the sense of Definition 1 in \S 4.1 of Lavendhomme \cite{l1}. The second
viewpoint goes as follows.

\begin{proposition}
\label{2.1}Tangent-vector-valued differential $p$-forms on $M$ can be
identified with mappings $K:D\times M^{D^{p}}\rightarrow M$ pursuant to the
following conditions:

\begin{enumerate}
\item $K(0,\gamma)=\mathfrak{o}_{p}(\gamma)$ for any $\gamma\in M^{D^{p}}$.

\item $K(\alpha d,\gamma)=K(d,\alpha\underset{i}{\cdot}\gamma)$ for any $d\in
D$, any $\alpha\in\mathbb{R}$, any $\gamma\in M^{D^{p}}$ and any natural
number $i$ with $1\leq i\leq p$.

\item $K(d,\gamma^{\sigma})=K(\varepsilon_{\sigma}d,\gamma)$ for any $d\in D
$, any $\gamma\in M^{D^{p}}$ and any $\sigma\in\mathbb{S}_{p}$.
\end{enumerate}
\end{proposition}

\begin{proof}
This follows from the set-theoretical identity
\[
(M^{D})^{M^{D^{p}}}=M^{D\times M^{D^{p}}}%
\]
The details can safely be left to the reader.
\end{proof}

Given $\varphi\in M^{M^{D^{p}}}$ and $\alpha\in\mathbb{R}$, we define
$\alpha\underset{i}{\cdot}\varphi\in M^{M^{D^{p}}}$ ($1\leq i\leq p$) to be
\[
(\alpha\underset{i}{\cdot}\varphi)(\gamma)=\varphi(\alpha\underset{i}{\cdot
}\gamma)
\]
for any $\gamma\in M^{D^{p}}$. Given $\varphi\in M^{M^{D^{p}}}$ and any
$\sigma\in\mathbb{S}_{p}$, we define $\varphi^{\sigma}\in M^{M^{D^{p}}}$ to
be
\[
\varphi^{\sigma}(\gamma)=\varphi(\gamma^{\sigma})
\]
for any $\gamma\in M^{D^{p}}$. Given $\varphi\in M^{M^{D^{p}}}$ and
$\sigma,\tau\in\mathbb{S}_{p}$, it is easy to see that
\[
\varphi^{\sigma\tau}(\gamma)=\varphi(\gamma^{\sigma\tau})=\varphi
((\gamma^{\sigma})^{\tau})=\varphi^{\tau}(\gamma^{\sigma})=(\varphi^{\tau
})^{\sigma}(\gamma)
\]
for any $\gamma\in M^{D^{p}}$, so that $\varphi^{\sigma\tau}=(\varphi^{\tau
})^{\sigma}$. The third viewpoint goes as follows.

\begin{proposition}
\label{2.2}Tangent-vector-valued differential $p$-forms on $M$ can be
identified with mappings $K:D\rightarrow M^{M^{D^{p}}}$ satisfying the
following conditions:

\begin{enumerate}
\item $K_{0}=\mathfrak{o}_{p}$

\item $\alpha\underset{i}{\cdot}K_{d}=K_{\alpha d}$ for any $d\in D$, any
$\alpha\in\mathbb{R}$ and any natural number $i$ with $1\leq i\leq p$.

\item $(K_{d})^{\sigma}=K_{\varepsilon_{\sigma}d}$ for any $d\in D$ and any
$\sigma\in\mathbb{S}_{p}$.
\end{enumerate}
\end{proposition}

\begin{proof}
This follows from the set-theoretical identity
\[
(M^{D})^{M^{D^{p}}}=(M^{M^{D^{p}}})^{D}%
\]
The details can safely be left to the reader.
\end{proof}

We will use the above three viewpoints on tangent-vector-valued differential
forms interchangeably, though we prefer the last one to the preceding two. The
following lemma, which will be used in the next section, should be obvious.

\begin{lemma}
\label{2.3}For any mappings $K,L:D^{2}\rightarrow M^{M^{D^{p}}}$ with
\begin{align*}
K(d,0) &  =L(d,0)\\
K(0,d) &  =L(0,d)
\end{align*}
for any $d\in D$, we have
\begin{align*}
\{(d_{1},d_{2}) &  \in D^{2}\mapsto L(d_{1},d_{2})^{\sigma}\}\overset{\cdot
}{-}\\
\{(d_{1},d_{2}) &  \in D^{2}\mapsto K(d_{1},d_{2})^{\sigma}\}\\
&  =d\in D\mapsto((L\overset{\cdot}{-}K)_{d})^{\sigma}%
\end{align*}
for any $\sigma\in\mathbb{S}_{p}$. Similar formulas hold for $\underset
{i}{\overset{\cdot}{-}}$ ($i=1,2,3$).
\end{lemma}

We will write $\Omega^{k}(M;\mathbf{T}M)$ for the totality of
tangent-vector-valued differential $k$-forms on $M$. Given $K\in\Omega
^{k}(M;\mathbf{T}M)$, $K^{\prime}\in\Omega^{k}(N;\mathbf{T}N)$ and
$f:M\rightarrow N$, we say that $K^{\prime}$ is $f$\textit{-related} to $K$ if
we have
\[
K_{d}^{\prime}(f\circ\gamma)=f(K_{d}(\gamma))
\]
for any $d\in D$ and any $\gamma\in M^{D}$.

If we drop the condition of the alternating property while keeping the
$k$-homogeneity in the definition of a tangent-vector-valued differential
$k$-form on $M$, we get the notion of a tangent-vector-valued differential
$k$-semiform on $M$. We denote by $\widehat{\Omega}^{k}(M;\mathbf{T}M)$ the
totality of tangent-vector-valued differential $k$-semiforms on $M$. Given
$K\in\widehat{\Omega}^{k}(M;\mathbf{T}M)$, we define $\mathcal{A}K\in
\Omega^{k}(M;\mathbf{T}M)$ to be
\[
\mathcal{A}K(\gamma)=\sum_{\sigma\in\mathbb{S}_{k}}\varepsilon_{\sigma
}K(\gamma^{\sigma})
\]
for any $\gamma\in M^{D^{k}}$, where $\varepsilon_{\sigma}$ is the sign of
$\sigma$. We write $\mathcal{A}_{p,q}K$ and $\mathcal{A}_{p,q,r}K$ for
$(1/p!q!)\mathcal{A}K$ and $(1/p!q!r!)\mathcal{A}K$ respectively.

\section{Interior Derivations}

Given $K\in\Omega^{k+1}(M;\mathbf{T}M)$ and $L\in\Omega^{l}(M;\mathbf{T}M)$,
we define $\widehat{\mathbf{i}}_{K}L\in\widehat{\Omega}^{k+l}(M;\mathbf{T}M)$
to be
\begin{align*}
(\widehat{\mathbf{i}}_{K}L)(\gamma)  &  =L\{(e_{1},...,e_{l})\in D^{l}\mapsto
K_{e_{1}}((d_{1},...d_{k+1})\in D^{k+1}\\
&  \mapsto\gamma(d_{1},...,d_{k+1},e_{2},...,e_{l}))\}
\end{align*}
for any $\gamma\in M^{D^{k+l}}$. Obviously we have to verify that

\begin{proposition}
We have
\[
\widehat{\mathbf{i}}_{K}L\in\widehat{\Omega}^{k+l}(M;\mathbf{T}M)
\]

\end{proposition}

\begin{proof}
Let $e$ be an arbitrary element of $D$ with $\alpha\in\mathbb{R}$. For $1\leq
i\leq k+1$, we have
\begin{align*}
&  (\widehat{\mathbf{i}}_{K}L)_{e}(\alpha\underset{i}{\cdot}\gamma)\\
&  =L_{e}\{(e_{1},...,e_{l})\in D^{l}\mapsto\\
K_{e_{1}}((d_{1},...d_{k+1})  &  \in D^{k+1}\mapsto\gamma(d_{1},...,,\alpha
d_{i},...,d_{k+1},e_{2},...,e_{l}))\}\\
&  =L_{e}\{(e_{1},...,e_{l})\in D^{l}\mapsto\\
K_{\alpha e_{1}}((d_{1},...d_{k+1})  &  \in D^{k+1}\mapsto\gamma
(d_{1},...,,d_{i},...,d_{k+1},e_{2},...,e_{l}))\}\\
&  =L_{\alpha e}\{(e_{1},...,e_{l})\in D^{l}\mapsto\\
K_{e_{1}}((d_{1},...d_{k+1})  &  \in D^{k+1}\mapsto\gamma(d_{1},...,,d_{i}%
,...,d_{k+1},e_{2},...,e_{l}))\}\\
&  =(\widehat{\mathbf{i}}_{K}L)_{\alpha e}(\gamma)
\end{align*}
while the case of $k+2\leq i\leq k+l$ can safely be left to the reader.
\end{proof}

Given $K\in\Omega^{k+1}(M;\mathbf{T}M)$ and $L\in\Omega^{l}(M;\mathbf{T}M)$,
we define $\mathbf{i}_{K}L\in\Omega^{k+l}(M;\mathbf{T}M)$ to be
\[
\mathcal{A}_{k+1,l-1}\left(  \widehat{\mathbf{i}}_{K}L\right)
\]

\begin{proposition}
Let $f:M\rightarrow N$ be a mapping. Let us suppose that $K^{\prime}\in
\Omega^{k+1}(N;\mathbf{T}N)$ is $f$-related to $K\in\Omega^{k+1}%
(M;\mathbf{T}M)$ and that $L^{\prime}\in\Omega^{l}(N;\mathbf{T}N)$ is
$f$-related to $L\in\Omega^{l}(M;\mathbf{T}M)$. Then $\mathbf{i}_{K^{\prime}%
}L^{\prime}$ is $f$-related to $\mathbf{i}_{K}L$.
\end{proposition}

\begin{proof}
It suffices to show that $\widehat{\mathbf{i}}_{K^{\prime}}L^{\prime}$ is
$f$-related to $\widehat{\mathbf{i}}_{K}L$. Let $d\in D$ and $\gamma\in
M^{D^{k+l}}$. Then we have
\begin{align*}
&  f((\widehat{\mathbf{i}}_{K}L)_{e}(\gamma))\\
&  =f[L_{e}\{(e_{1},...,e_{l})\in D^{l}\mapsto\\
K_{e_{1}}((d_{1},...d_{k+1})  &  \in D^{k+1}\mapsto\gamma(d_{1},...,d_{k+1}%
,e_{2},...,e_{l}))\}]\\
&  =L_{e}^{\prime}[(e_{1},...,e_{l})\in D^{l}\mapsto\\
f\{K_{e_{1}}((d_{1},...d_{k+1})  &  \in D^{k+1}\mapsto\gamma(d_{1}%
,...,d_{k+1},e_{2},...,e_{l}))\}]\\
&  =L_{e}^{\prime}[(e_{1},...,e_{l})\in D^{l}\mapsto K_{d_{1}}^{\prime
}((f\circ\gamma)(\cdot_{1},...,\cdot_{k+1},d_{2},...,d_{l}))\\
K_{e_{1}}\{(d_{1},...d_{k+1})  &  \in D^{k+1}\mapsto(f\circ\gamma
)(d_{1},...,d_{k+1},e_{2},...,e_{l})\}]\\
&  =(\widehat{\mathbf{i}}_{K^{\prime}}L^{\prime})_{e}(f\circ\gamma)
\end{align*}
which completes the proof.
\end{proof}

\section{The Fr\"{o}licher-Nijenhuis Bracket}

Given $\varphi\in M^{M^{D^{p}}}$ and $\psi\in M^{M^{D^{q}}}$, we define
$\psi\ast\varphi\in M^{M^{D^{p+q}}}$ to be
\begin{align*}
&  \psi\ast\varphi(\gamma)\\
&  =\psi\{(e_{1},...,e_{q})\in D^{q}\mapsto\\
\varphi((d_{1},...,d_{p})  &  \in D^{p}\mapsto\gamma(d_{1},...,d_{p}%
,e_{1},...,e_{q}))\}
\end{align*}
for any $\gamma\in M^{D^{p+q}}$. Given two tangent-vector-valued differential
forms $K:D\rightarrow M^{M^{D^{p}}}$ and $L:D\rightarrow M^{M^{D^{q}}}$, we
define a mapping $L\ast K:D^{2}\rightarrow M^{M^{D^{p+q}}} $ to be
\[
(L\ast K)(d_{1},d_{2})=L_{d_{2}}\ast K_{d_{1}}%
\]
for any $(d_{1},d_{2})\in D^{2}$. The following lemma should be obvious.

\begin{lemma}
\label{t4.1}Given two tangent-vector-valued differential forms $K:D\rightarrow
M^{M^{D^{p}}}$ and $L:D\rightarrow M^{M^{D^{q}}}$ with $\sigma=(
\begin{array}
[c]{cccccc}%
1 & ... & q & q+1 & ... & p+q\\
p+1 & ... & p+q & 1 & ... & p
\end{array}
)\in\mathbb{S}_{p+q}$, we have

\begin{enumerate}
\item $(L\ast K)(d,0)=((K\ast L)(0,d))^{\sigma}$ for any $d\in D$.

\item $(L\ast K)(0,d)=((K\ast L)(d,0))^{\sigma}$ for any $d\in D$.
\end{enumerate}
\end{lemma}

We continue to use the notation of the above lemma for a while. We denote by
$K\widetilde{\ast}L$ the mapping $(d_{1},d_{2})\in D^{2}\mapsto((K\ast
L)(d_{2},d_{1}))^{\sigma}\in M^{M^{D^{p+q}}}$. We are thus entitled by the
above lemma to define $\left\lfloor K,L\right\rfloor \in(M^{M^{D^{p+q}}})^{D}$
to be
\[
L\ast K\overset{\cdot}{-}K\widetilde{\ast}L
\]

\begin{lemma}
\label{t4.2}The mapping $\left\lfloor K,L\right\rfloor :D\rightarrow
M^{M^{D^{p+q}}}$ satisfies the following conditions:

\begin{enumerate}
\item $\left\lfloor K,L\right\rfloor _{0}=\mathfrak{o}_{p}$

\item $\alpha\underset{i}{\cdot}\left\lfloor K,L\right\rfloor _{d}%
=\left\lfloor K,L\right\rfloor _{\alpha d}$ for any $d\in D$, any $\alpha
\in\mathbb{R}$ and any natural number $i$ with $1\leq i\leq p+q$.
\end{enumerate}
\end{lemma}

\begin{proof}
The first condition should be obvious. To see the second condition, we note that

\begin{enumerate}
\item $\alpha\underset{i}{\cdot}(L\ast K)(d_{1},d_{2})=(L\ast K)(\alpha
d_{1},d_{2})$ and $\alpha\underset{i}{\cdot}(K\widetilde{\ast}L)(d_{1}%
,d_{2})=(K\widetilde{\ast}L)(\alpha d_{1},d_{2})$ for any natural number $i$
with $1\leq i\leq p$.

\item $\alpha\underset{i}{\cdot}(L\ast K)(d_{1},d_{2})=(L\ast K)(d_{1},\alpha
d_{2})$ and $\alpha\underset{i}{\cdot}(K\widetilde{\ast}L)(d_{1}%
,d_{2})=(K\widetilde{\ast}L)(d_{1},\alpha d_{2})$ for any natural number $i$
with $p+1\leq i\leq p+q$.
\end{enumerate}

Therefore the second condition follows by Proposition 5 in \S 3.4 of
Lavendhomme \cite{l1} from the first property in case of $1\leq i\leq p$ and
from the second property in case of $p+1\leq i\leq p+q$.
\end{proof}

\begin{lemma}
\label{t4.3}Given three tangent-vector-valued differential forms
$K_{1}:D\rightarrow M^{M^{D^{p}}}$, $K_{2}:D\rightarrow M^{M^{D^{q}}}$ and
$K_{3}:D\rightarrow M^{M^{D^{r}}}$, we have
\[
\mathcal{A}_{p,q+r}(\left\lfloor K_{1},\mathcal{A}_{q,r}(\left\lfloor
K_{2},K_{3}\right\rfloor )\right\rfloor )=\mathcal{A}_{p,q,r}(\left\lfloor
K_{1},\left\lfloor K_{2},K_{3}\right\rfloor \right\rfloor )
\]

\end{lemma}

\begin{proof}
By the same token as in the familiar associativity of wedge products in
differential forms.
\end{proof}

We are going to define the \textit{Fr\"{o}licher-Nijenhuis bracket}
$\left\lceil K,L\right\rceil $ to be
\[
\left\lceil K,L\right\rceil =\mathcal{A}_{p,q}\mathcal{(}\left\lfloor
K,L\right\rfloor )
\]
which is undoubtedly a tangent-vector-valued differential $(p+q)$-form.

\begin{theorem}
\label{t4.4}The following two properties hold for the Fr\"{o}licher-Nijenhuis bracket:

\begin{enumerate}
\item We have
\[
\left\lceil K,L\right\rceil =-(-1)^{pq}\left\lceil L,K\right\rceil
\]
for any two tangent-vector-valued differential forms $K:D\rightarrow
M^{M^{D^{p}}}$ and $L:D\rightarrow M^{M^{D^{q}}}$.

\item We have
\[
\left\lceil K_{1},\left\lceil K_{2},K_{3}\right\rceil \right\rceil
+(-1)^{p(q+r)}\left\lceil K_{2},\left\lceil K_{3},K_{1}\right\rceil
\right\rceil +(-1)^{r(p+q)}\left\lceil K_{3},\left\lceil K_{1},K_{2}%
\right\rceil \right\rceil =0
\]
for any three tangent-vector-valued differential forms $K_{1}:D\rightarrow
M^{M^{D^{p}}}$, $K_{2}:D\rightarrow M^{M^{D^{q}}}$ and $K_{3}:D\rightarrow
M^{M^{D^{r}}}$.
\end{enumerate}
\end{theorem}

\begin{proof}
In order to see the first property, it suffices to note that
\begin{align*}
(L\ast K)(d_{1},d_{2})^{\sigma} &  =(L\widetilde{\ast}K)(d_{2},d_{1})\\
(K\widetilde{\ast}L)(d_{1},d_{2})^{\sigma} &  =(K\ast L)(d_{2},d_{1})
\end{align*}
from which it follows Propositions 4 and 6 in \S 3.4 of Lavendhomme \cite{l1}
and Lemma \ref{2.3} that
\begin{align*}
&  \left\lceil K,L\right\rceil \\
&  =\mathcal{A}_{p,q}\mathcal{(}\left\lfloor K,L\right\rfloor )\\
&  =\frac{1}{p!q!}\sum_{\tau\in\mathbb{S}_{p+q}}\varepsilon_{\tau}%
\varepsilon_{\sigma}\{d\in D\mapsto((L\ast K\overset{\cdot}{-}K\widetilde
{\ast}L)_{d})^{\tau\sigma}\}\\
&  =\frac{1}{p!q!}\varepsilon_{\sigma}\sum_{\tau\in\mathbb{S}_{p+q}%
}\varepsilon_{\tau}[\{(d_{1},d_{2})\in D^{2}\mapsto(L\ast K)(d_{1}%
,d_{2})^{\tau\sigma}\}\overset{\cdot}{-}\\
\{(d_{1},d_{2}) &  \in D^{2}\mapsto(K\widetilde{\ast}L)(d_{1},d_{2}%
)^{\tau\sigma}\}]\\
&  =\frac{1}{p!q!}\varepsilon_{\sigma}\sum_{\tau\in\mathbb{S}_{p+q}%
}\varepsilon_{\tau}[\{(d_{1},d_{2})\in D^{2}\mapsto((L\ast K)(d_{1}%
,d_{2})^{\sigma})^{\tau}\}\overset{\cdot}{-}\\
\{(d_{1},d_{2}) &  \in D^{2}\mapsto((K\widetilde{\ast}L)(d_{1},d_{2})^{\sigma
})^{\tau}\}]\\
&  =\frac{1}{p!q!}\varepsilon_{\sigma}\sum_{\tau\in\mathbb{S}_{p+q}%
}\varepsilon_{\tau}[\{(d_{1},d_{2})\in D^{2}\mapsto(L\widetilde{\ast}%
K)(d_{2},d_{1})^{\tau}\}\overset{\cdot}{-}\\
\{(d_{1},d_{2}) &  \in D^{2}\mapsto(K\ast L)(d_{2},d_{1})^{\tau}\}]\\
&  =\frac{1}{p!q!}\varepsilon_{\sigma}\sum_{\tau\in\mathbb{S}_{p+q}%
}\varepsilon_{\tau}\{d\in D\mapsto((L\widetilde{\ast}K\overset{\cdot}{-}K\ast
L)_{d})^{\tau}\}\\
&  =-\frac{1}{p!q!}\varepsilon_{\sigma}\sum_{\tau\in\mathbb{S}_{p+q}%
}\varepsilon_{\tau}\{d\in D\mapsto((K\ast L\overset{\cdot}{-}L\widetilde{\ast
}K)_{d})^{\tau}\}\\
&  =-\frac{1}{p!q!}\varepsilon_{\sigma}\sum_{\tau\in\mathbb{S}_{p+q}%
}\varepsilon_{\tau}\{d\in D\mapsto(\left\lfloor L,K\right\rfloor _{d})^{\tau
}\}\\
&  =-\varepsilon_{\sigma}\left\lceil L,K\right\rceil
\end{align*}
Since it is easy to see that $\varepsilon_{\sigma}=(-1)^{pq}$, the desired
first property follows at once. In order to see the second property, we first
define six mappings $\varphi_{123},\varphi_{132},\varphi_{213},\varphi
_{231},\varphi_{312},\varphi_{321}:D^{3}\rightarrow M^{M^{D^{p+q+r}}}$ to be
\begin{align*}
\varphi_{123} &  =(d_{1},d_{2},d_{3})\in D^{3}\mapsto(K_{3})_{d_{3}}\ast
(K_{2})_{d_{2}}\ast(K_{1})_{d_{1}}\in M^{M^{D^{p+q+r}}}\\
\varphi_{132} &  =(d_{1},d_{2},d_{3})\in D^{3}\mapsto((K_{2})_{d_{2}}%
\ast(K_{3})_{d_{3}}\ast(K_{1})_{d_{1}})^{\sigma_{132}}\in M^{M^{D^{p+q+r}}}\\
\varphi_{213} &  =(d_{1},d_{2},d_{3})\in D^{3}\mapsto((K_{3})_{d_{3}}%
\ast(K_{1})_{d_{1}}\ast(K_{2})_{d_{2}})^{\sigma_{213}}\in M^{M^{D^{p+q+r}}}\\
\varphi_{231} &  =(d_{1},d_{2},d_{3})\in D^{3}\mapsto((K_{1})_{d_{1}}%
\ast(K_{3})_{d_{3}}\ast(K_{2})_{d_{2}})^{\sigma_{231}}\in M^{M^{D^{p+q+r}}}\\
\varphi_{312} &  =(d_{1},d_{2},d_{3})\in D^{3}\mapsto((K_{2})_{d_{2}}%
\ast(K_{1})_{d_{1}}\ast(K_{3})_{d_{3}})^{\sigma_{312}}\in M^{M^{D^{p+q+r}}}\\
\varphi_{321} &  =(d_{1},d_{2},d_{3})\in D^{3}\mapsto((K_{1})_{d_{1}}%
\ast(K_{2})_{d_{2}}\ast(K_{3})_{d_{3}})^{\sigma_{321}}\in M^{M^{D^{p+q+r}}}%
\end{align*}
where
\begin{align*}
\sigma_{132} &  =\left(
\begin{array}
[c]{ccccccccc}%
1 & ... & p & p+1 & ... & p+r & p+r+1 & ... & p+q+r\\
1 & ... & p & p+q+1 & ... & p+q+r & p+1 & ... & p+q
\end{array}
\right) \\
\sigma_{213} &  =\left(
\begin{array}
[c]{ccccccccc}%
1 & ... & q & q+1 & ... & p+q & p+q+1 & ... & p+q+r\\
p+1 & ... & p+q & 1 & ... & p & p+q+1 & ... & p+q+r
\end{array}
\right) \\
\sigma_{231} &  =\left(
\begin{array}
[c]{ccccccccc}%
1 & ... & q & q+1 & ... & q+r & q+r+1 & ... & p+q+r\\
p+1 & ... & p+q & p+q+1 & ... & p+q+r & 1 & ... & p
\end{array}
\right) \\
\sigma_{312} &  =\left(
\begin{array}
[c]{ccccccccc}%
1 & ... & r & r+1 & ... & p+r & p+r+1 & ... & p+q+r\\
p+q+1 & ... & p+q+r & 1 & ... & p & p+1 & ... & p+q
\end{array}
\right) \\
\sigma_{321} &  =\left(
\begin{array}
[c]{ccccccccc}%
1 & ... & r & r+1 & ... & q+r & q+r+1 & ... & p+q+r\\
p+q+1 & ... & p+q+r & p+1 & ... & p+q & 1 & ... & p
\end{array}
\right)
\end{align*}
Now we have
\begin{align*}
&  \left\lceil K_{1},\left\lceil K_{2},K_{3}\right\rceil \right\rceil \\
&  =\mathcal{A}_{p,q+r}(\left\lfloor K_{1},\mathcal{A}_{q,r}(\left\lfloor
K_{2},K_{3}\right\rfloor )\right\rfloor )\\
&  =\mathcal{A}_{p,q,r}(\left\lfloor K_{1},\left\lfloor K_{2},K_{3}%
\right\rfloor \right\rfloor )\\
&  =\mathcal{A}_{p,q,r}((\varphi_{123}\overset{\cdot}{\underset{1}{-}}%
\varphi_{132})\overset{\cdot}{-}(\varphi_{231}\overset{\cdot}{\underset{1}{-}%
}\varphi_{321}))
\end{align*}
Let $\rho_{1}\in\mathbb{S}_{p+q+r}$ be $\sigma_{231}$, for which we have
$\varepsilon_{\rho_{1}}=(-1)^{p(q+r)}$. Let $\varphi_{231}^{2},\varphi
_{213}^{2},\varphi_{312}^{2},\varphi_{132}^{2}:D^{3}\mapsto M^{M^{D^{p+q+r}}}$
be mapping
\begin{align*}
\varphi_{231}^{2} &  =(d_{1},d_{2},d_{3})\in D^{3}\mapsto\varphi_{231}%
(d_{1},d_{2},d_{3})^{\rho_{1}}\\
\varphi_{213}^{2} &  =(d_{1},d_{2},d_{3})\in D^{3}\mapsto\varphi_{213}%
(d_{1},d_{2},d_{3})^{\rho_{1}}\\
\varphi_{312}^{2} &  =(d_{1},d_{2},d_{3})\in D^{3}\mapsto\varphi_{312}%
(d_{1},d_{2},d_{3})^{\rho_{1}}\\
\varphi_{132}^{2} &  =(d_{1},d_{2},d_{3})\in D^{3}\mapsto\varphi_{132}%
(d_{1},d_{2},d_{3})^{\rho_{1}}%
\end{align*}
Now we have
\begin{align*}
&  \left\lceil K_{2},\left\lceil K_{3},K_{1}\right\rceil \right\rceil \\
&  =\mathcal{A}_{q,r+p}(\left\lfloor K_{2},\mathcal{A}_{r,p}(\left\lfloor
K_{3},K_{1}\right\rfloor )\right\rfloor )\\
&  =\mathcal{A}_{q,r,p}(\left\lfloor K_{2},\left\lfloor K_{3},K_{1}%
\right\rfloor \right\rfloor )\\
&  =\mathcal{A}_{q,r,p}((\varphi_{231}^{2}\overset{\cdot}{\underset{2}{-}%
}\varphi_{213}^{2})\overset{\cdot}{-}(\varphi_{312}^{2}\overset{\cdot
}{\underset{2}{-}}\varphi_{132}^{2}))\\
&  =\frac{1}{p!q!r!}\sum_{\tau\in\mathbb{S}_{p+q+r}}\varepsilon_{\tau}\{d\in
D\mapsto(((\varphi_{231}\overset{\cdot}{\underset{2}{-}}\varphi_{213}%
)\overset{\cdot}{-}(\varphi_{312}\overset{\cdot}{\underset{2}{-}}\varphi
_{132}))_{d})^{\tau\rho_{1}}\}\\
&  =\frac{1}{p!q!r!}\varepsilon_{\rho_{1}}\sum_{\tau\in\mathbb{S}_{p+q+r}%
}\varepsilon_{\tau}\varepsilon_{\rho_{1}}\{d\in D\mapsto(((\varphi
_{231}\overset{\cdot}{\underset{2}{-}}\varphi_{213})\overset{\cdot}{-}%
(\varphi_{312}\overset{\cdot}{\underset{2}{-}}\varphi_{132}))_{d})^{\tau
\rho_{1}}\}\\
&  =\varepsilon_{\rho_{1}}\mathcal{A}_{q,r,p}((\varphi_{231}\overset{\cdot
}{\underset{2}{-}}\varphi_{213})\overset{\cdot}{-}(\varphi_{312}\overset
{\cdot}{\underset{2}{-}}\varphi_{132}))
\end{align*}
which implies that
\begin{align*}
&  (-1)^{p(q+r)}\left\lceil K_{2},\left\lceil K_{3},K_{1}\right\rceil
\right\rceil \\
&  =\mathcal{A}_{q,r,p}((\varphi_{231}\overset{\cdot}{\underset{2}{-}}%
\varphi_{213})\overset{\cdot}{-}(\varphi_{312}\overset{\cdot}{\underset{2}{-}%
}\varphi_{132}))
\end{align*}
Let $\rho_{2}\in\mathbb{S}_{p+q+r}$ be $\sigma_{312}$, for which we have
$\varepsilon_{\rho_{2}}=(-1)^{r(p+q)}$. Let $\varphi_{312}^{3},\varphi
_{321}^{3},\varphi_{123}^{3},\varphi_{213}^{3}:D^{3}\mapsto M^{M^{D^{p+q+r}}}$
be mappings
\begin{align*}
\varphi_{312}^{3} &  =(d_{1},d_{2},d_{3})\in D^{3}\mapsto\varphi_{312}%
(d_{1},d_{2},d_{3})^{\rho_{2}}\\
\varphi_{321}^{3} &  =(d_{1},d_{2},d_{3})\in D^{3}\mapsto\varphi_{321}%
(d_{1},d_{2},d_{3})^{\rho_{2}}\\
\varphi_{123}^{3} &  =(d_{1},d_{2},d_{3})\in D^{3}\mapsto\varphi_{123}%
(d_{1},d_{2},d_{3})^{\rho_{2}}\\
\varphi_{213}^{3} &  =(d_{1},d_{2},d_{3})\in D^{3}\mapsto\varphi_{213}%
(d_{1},d_{2},d_{3})^{\rho_{2}}%
\end{align*}
Now we have
\begin{align*}
&  \left\lceil K_{3},\left\lceil K_{1},K_{2}\right\rceil \right\rceil \\
&  =\mathcal{A}_{r,p+q}(\left\lfloor K_{3},\mathcal{A}_{p,q}(\left\lfloor
K_{1},K_{2}\right\rfloor )\right\rfloor )\\
&  =\mathcal{A}_{r,p,q}(\left\lfloor K_{3},\left\lfloor K_{1},K_{2}%
\right\rfloor \right\rfloor )\\
&  =\mathcal{A}_{r,p,q}((\varphi_{312}^{3}\overset{\cdot}{\underset{3}{-}%
}\varphi_{321}^{3})\overset{\cdot}{-}(\varphi_{123}^{3}\overset{\cdot
}{\underset{3}{-}}\varphi_{213}^{3}))\\
&  =\frac{1}{p!q!r!}\sum_{\tau\in\mathbb{S}_{p+q+r}}\varepsilon_{\tau}\{d\in
D\mapsto(((\varphi_{312}\overset{\cdot}{\underset{3}{-}}\varphi_{321}%
)\overset{\cdot}{-}(\varphi_{123}\overset{\cdot}{\underset{3}{-}}\varphi
_{213}))_{d})^{\tau\rho_{2}}\}\\
&  =\frac{1}{p!q!r!}\varepsilon_{\rho_{2}}\sum_{\tau\in\mathbb{S}_{p+q+r}%
}\varepsilon_{\tau}\varepsilon_{\rho_{2}}\{d\in D\mapsto(((\varphi
_{312}\overset{\cdot}{\underset{3}{-}}\varphi_{321})\overset{\cdot}{-}%
(\varphi_{123}\overset{\cdot}{\underset{3}{-}}\varphi_{213}))_{d})^{\tau
\rho_{2}}\}\\
&  =\varepsilon_{\rho_{2}}\mathcal{A}_{r,p,q}((\varphi_{312}\overset{\cdot
}{\underset{3}{-}}\varphi_{321})\overset{\cdot}{-}(\varphi_{123}\overset
{\cdot}{\underset{3}{-}}\varphi_{213}))
\end{align*}
which implies that
\begin{align*}
&  (-1)^{r(p+q)}\left\lceil K_{3},\left\lceil K_{1},K_{2}\right\rceil
\right\rceil \\
&  =\mathcal{A}_{r,p,q}((\varphi_{312}\overset{\cdot}{\underset{3}{-}}%
\varphi_{321})\overset{\cdot}{-}(\varphi_{123}\overset{\cdot}{\underset{3}{-}%
}\varphi_{213}))
\end{align*}
Therefore we have
\begin{align*}
&  \left\lceil K_{1},\left\lceil K_{2},K_{3}\right\rceil \right\rceil
+(-1)^{p(q+r)}\left\lceil K_{2},\left\lceil K_{3},K_{1}\right\rceil
\right\rceil +(-1)^{r(p+q)}\left\lceil K_{3},\left\lceil K_{1},K_{2}%
\right\rceil \right\rceil \\
&  =\mathcal{A}_{p,q,r}((\varphi_{123}\overset{\cdot}{\underset{1}{-}}%
\varphi_{132})\overset{\cdot}{-}(\varphi_{231}\overset{\cdot}{\underset{1}{-}%
}\varphi_{321}))+\\
&  \mathcal{A}_{q,r,p}((\varphi_{231}\overset{\cdot}{\underset{2}{-}}%
\varphi_{213})\overset{\cdot}{-}(\varphi_{312}\overset{\cdot}{\underset{2}{-}%
}\varphi_{132}))+\\
&  \mathcal{A}_{r,p,q}((\varphi_{312}\overset{\cdot}{\underset{3}{-}}%
\varphi_{321})\overset{\cdot}{-}(\varphi_{123}\overset{\cdot}{\underset{3}{-}%
}\varphi_{213}))\\
&  =\frac{1}{p!q!r!}\mathcal{A(\{}(\varphi_{123}\overset{\cdot}{\underset
{1}{-}}\varphi_{132})\overset{\cdot}{-}(\varphi_{231}\overset{\cdot}%
{\underset{1}{-}}\varphi_{321})\}+\{(\varphi_{231}\overset{\cdot}{\underset
{2}{-}}\varphi_{213})\overset{\cdot}{-}(\varphi_{312}\overset{\cdot}%
{\underset{2}{-}}\varphi_{132})\}\\
&  +\{(\varphi_{312}\overset{\cdot}{\underset{3}{-}}\varphi_{321}%
)\overset{\cdot}{-}(\varphi_{123}\overset{\cdot}{\underset{3}{-}}\varphi
_{213})\})\\
&  =0\text{ \ \ \ [by the general Jacobi identity]}%
\end{align*}

\end{proof}

Now we are going to show the naturality of the Fr\"{o}licher-Nijenhuis
bracket. Let $f:M\rightarrow N$ be a function of microlinear spaces.

\begin{lemma}
\label{t4.5}If tangent-vector-valued differential forms $K^{\prime
}:D\rightarrow N^{N^{D^{p}}}$ and $L^{\prime}:D\rightarrow N^{N^{D^{q}}}$ are
$f$-related to tangent-vector-valued differential forms $K:D\rightarrow
M^{M^{D^{p}}}$ and $L:D\rightarrow M^{M^{D^{q}}}$ respectively, then we have
\begin{align*}
f\circ\left(  (L\ast K)(\gamma)\right)   &  =(L^{\prime}\ast K^{\prime
})(f\circ\gamma)\\
f\circ\left(  (K\widetilde{\ast}L)(\gamma)\right)   &  =(K^{\prime}%
\widetilde{\ast}L^{\prime})(f\circ\gamma)
\end{align*}
for any $\gamma\in M^{D^{p+q}}$.
\end{lemma}

\begin{proof}
For the first identity, we have
\begin{align*}
&  f\circ\left(  (L\ast K)(\gamma)\right) \\
&  =f\circ\lbrack(d,d^{\prime})\in D^{2}\mapsto L_{d^{\prime}}\{(e_{1}%
,...,e_{q})\in D^{q}\mapsto\\
K_{d}((d_{1},...,d_{p})  &  \in D^{p}\mapsto\gamma(d_{1},...,d_{p}%
,e_{1},...,e_{q}))\}]\\
&  =(d,d^{\prime})\in D^{2}\mapsto f[L_{d^{\prime}}\{(e_{1},...,e_{q})\in
D^{q}\mapsto\\
K_{d}((d_{1},...,d_{p})  &  \in D^{p}\mapsto\gamma(d_{1},...,d_{p}%
,e_{1},...,e_{q}))\}]\\
&  =(d,d^{\prime})\in D^{2}\mapsto L_{d^{\prime}}^{\prime}[(e_{1}%
,...,e_{q})\in D^{q}\mapsto\\
f\{K_{d}((d_{1},...,d_{p})  &  \in D^{p}\mapsto\gamma(d_{1},...,d_{p}%
,e_{1},...,e_{q}))\}]\\
&  =(d,d^{\prime})\in D^{2}\mapsto L_{d^{\prime}}^{\prime}[(e_{1}%
,...,e_{q})\in D^{q}\mapsto\\
\{K_{d}((d_{1},...,d_{p})  &  \in D^{p}\mapsto(f\circ\gamma)(d_{1}%
,...,d_{p},e_{1},...,e_{q}))\}]\\
&  =(L^{\prime}\ast K^{\prime})(f\circ\gamma)
\end{align*}
The second formula can be established by the same token.
\end{proof}

\begin{proposition}
\label{t4.6}Under the same assumption and notation as in the above lemma,
$\left\lceil K^{\prime},L^{\prime}\right\rceil $ is $f$-related to
$\left\lceil K,L\right\rceil $.
\end{proposition}

\begin{proof}
It suffices to show that $\left\lfloor K^{\prime},L^{\prime}\right\rfloor $ is
$f$-related to $\left\lfloor K,L\right\rfloor $, which follows from the
following calculation:
\begin{align*}
&  f\circ\left(  \left\lfloor K,L\right\rfloor (\gamma)\right) \\
&  =f\circ\{(L\ast K)(\gamma)\overset{\cdot}{-}(K\widetilde{\ast}%
L)(\gamma)\}\\
&  =f\circ\left(  (L\ast K)(\gamma)\right)  \overset{\cdot}{-}f\circ\left(
(K\widetilde{\ast}L)(\gamma)\right) \\
&  =(L^{\prime}\ast K^{\prime})(f\circ\gamma)\overset{\cdot}{-}(K^{\prime
}\widetilde{\ast}L^{\prime})(f\circ\gamma)\\
&  =\left\lfloor K^{\prime},L^{\prime}\right\rfloor (f\circ\gamma)
\end{align*}
for any $\gamma\in M^{D^{p+q}}$.
\end{proof}

\section{Lie Derivations}

Let $K\in\Omega^{k}(M;\mathbf{T}M)$ and $L\in\Omega^{l}(M;\mathbf{T}M)$. Let
$\nabla$ be a linear connection on $M$. It is easy to see that

\begin{lemma}
We have
\begin{align*}
(L\ast K)(d,0)  &  =\nabla(L\ast K)(d,0)\\
(L\ast K)(0,d)  &  =\nabla(L\ast K)(0,d)
\end{align*}
for any $d\in D$.
\end{lemma}

Now we define $\widehat{\mathbf{L}}_{K}^{\nabla}L\in\widehat{\Omega}%
^{k+l}(M;\mathbf{T}M)$ to be
\[
\widehat{\mathbf{L}}_{K}^{\nabla}L(\gamma)=(L\ast K)(\gamma)\overset{\cdot}%
{-}\nabla((L\ast K)(\gamma))
\]
for any $\gamma\in M^{D^{k+l}}$. Indeed we have to verify that

\begin{lemma}
We have
\[
\widehat{\mathbf{L}}_{K}^{\nabla}L(\alpha\underset{i}{\cdot}\gamma
)=\alpha\left(  \widehat{\mathbf{L}}_{K}^{\nabla}L(\gamma)\right)
\]
for any $\alpha\in\mathbb{R}$ and any natural number $i$ with $1\leq i\leq
k+l$.
\end{lemma}

\begin{proof}
By the same token as in the proof of Lemma \ref{t4.2}.
\end{proof}

\begin{proposition}
With the above notation, we have
\begin{multline*}
\widehat{\mathbf{L}}_{K}^{\nabla}L(\gamma)\\
=\mathbf{D[}e\in D\mapsto\mathbf{q}_{(t,e)}[L\{(e_{1},...,e_{l})\in
D^{l}\mapsto K_{e}((d_{1},...d_{k})\in D^{k}\\
\mapsto\gamma(d_{1},...,d_{k},e_{1},...,e_{l}))\}]]
\end{multline*}
for any $\gamma\in M^{D^{k+l}}$, where $t\in M^{D}$ is the mapping $d\in
D\mapsto K_{d}((d_{1},...d_{k})\in D^{k}\mapsto\gamma(d_{1},...,d_{k}%
,0,...,0))\in M$.
\end{proposition}

\begin{proof}
By Propositions 3 and 7 in \S \S 5.2 of Lavendhomme \cite{l1}.
\end{proof}

We define $\mathbf{L}_{K}^{\nabla}L\in\Omega^{k+l}(M;\mathbf{T}M)$ to be
\[
\mathbf{L}_{K}^{\nabla}L=\mathcal{A}_{k,l}\left(  \widehat{\mathbf{L}}%
_{K}^{\nabla}L\right)
\]

\begin{proposition}
Continuing with the above notation and assuming that the linear connection
$\nabla$ is symmetric, we have
\[
\left\lceil K,L\right\rceil =\mathbf{L}_{K}^{\nabla}L-(-1)^{kl}\mathbf{L}%
_{L}^{\nabla}K
\]

\end{proposition}

\begin{proof}
It suffices to show that
\[
\left\lfloor K,L\right\rfloor =\widehat{\mathbf{L}}_{K}^{\nabla}L-\left(
\widehat{\mathbf{L}}_{L}^{\nabla}K\right)  ^{\sigma}%
\]
with $\sigma=\left(
\begin{array}
[c]{cccccc}%
1 & ... & k & k+1 & ... & k+l\\
k+1 & ... & k+l & 1 & ... & k
\end{array}
\right)  \in\mathbb{S}_{k+l}$, which follows by the same token as in
Proposition 3 in \S \S 5.3 of Lavendhomme \cite{l1}.
\end{proof}

Finally we are going to establish the naturality of Lie derivations. Let
$f:M\rightarrow N$ be a function of microlinear spaces with a linear
connection $\nabla^{\prime}$ on $N$ being $f$-related to the linear connection
$\nabla$ on $M$.

\begin{lemma}
Let $K^{\prime}\in\Omega^{k}(N;\mathbf{T}N)$ and $L^{\prime}\in\Omega
^{l}(N;\mathbf{T}N)$ be $f$-related to $K\in\Omega^{k}(M;\mathbf{T}M)$ and
$L\in\Omega^{l}(M;\mathbf{T}M)$ respectively. Then we have
\[
f\circ\left(  \nabla(L\ast K)(\gamma)\right)  =\nabla^{\prime}(L^{\prime}\ast
K^{\prime})(f\circ\gamma)
\]
for any $\gamma\in M^{D^{k+l}}$.
\end{lemma}

\begin{proof}
By the same token as in the proof of Lemma \ref{t4.5}.
\end{proof}

\begin{proposition}
Under the same assumption and notation as in the above lemma, $\mathbf{L}%
_{K^{\prime}}^{\nabla^{\prime}}L^{\prime}$ is $f$-related to $\mathbf{L}%
_{K}^{\nabla}L$.
\end{proposition}

\begin{proof}
By the same token as in the proof of Proposition \ref{t4.6}.
\end{proof}


\begin{thebibliography}{99}                                                                                               %
\bibitem {fn}Fr\"{o}licher, A. and Nijenhuis, A.:Theory of vector-valued
differential forms, Part I, Indagationes Math., \textbf{18} (1956), 338-359.

\bibitem {kl}Kock, A. and Lavendhomme, R.:Strong infinitesimal linearity, with
applications to strong difference and affine connections, Cahiers de Topologie
et Geom\'{e}trie Differ\'{e}ntielle, \textbf{25} (1984), 311-324.

\bibitem {l1}Lavendhomme, R.: Basic Concepts of Synthetic Differential
Geometry, Kluwer, Dordrecht, 1996.

\bibitem {m0}Michor, Peter W.:Remarks on the Fr\"{o}licher-Nijenhuis bracket,
Proceedings of the Conference on Differential Geometry and its Applications,
Brno 1986, D. Reidel, 1987, pp.197-220.

\bibitem {m1}Michor, Peter W.:Topics in Differential Geometry, American
Mathematical Society, Providence, Rhode Islands, 2008.

\bibitem {m2}Minguez, M.C.:Wedge products of forms in synthetic differential
geometry, Cahiers de Topologie et Geom\'{e}trie Differ\'{e}ntielle,
\textbf{29} (1988), 59-66.

\bibitem {m3}Minguez, M.C.:Some combinatorial calculus on Lie derivatives,
Cahiers de Topologie et Geom\'{e}trie Differ\'{e}ntielle, \textbf{29} (1988), 241-247.

\bibitem {n0}Nijenhuis, A.:Jacobi type identities for bilinear differential
concomitants of certain tensor fields I, Indagationes Math., \textbf{17}
(1955), 390-403.

\bibitem {n1}Nishimura, H.:Theory of microcubes, International Journal of
Theoretical Physics, \textbf{36} (1997), 1099-1131.

\bibitem {n2}Nishimura, H.:General Jacobi identity revisited, International
Journal of Theoretical Physics, \textbf{38} (1999), 2163-2174.

\bibitem {n3}Nishimura, H.:General Jacobi identity revisited again,
International Journal of Theoretical Physics, \textbf{46} (2007), 2843-2862.
\end{thebibliography}
\end{document}